\documentclass[a4paper]{amsart}
\usepackage{amsmath,amsthm,amssymb}
\usepackage{enumerate}

\newcommand{\alg}{\mathcal{A}}
\newcommand{\anch}{\rho}
\newcommand{\cov}{\nabla}
\DeclareMathOperator{\End}{End}
\newcommand{\eps}{\epsilon}
\newcommand{\lie}{\mathcal{L}}
\newcommand{\R}{\mathbb{R}}

\newcommand{\shuffle}{\mathrm{Sh}}
\newcommand{\smooth}{\mathcal{C}^\infty}

\newcommand{\wdd}{\overline{\wedge}}

\newtheorem{theorem}{Theorem}
\newtheorem{proposition}[theorem]{Proposition}

\theoremstyle{Definition}
\newtheorem{definition}[theorem]{Definition}
\title[Covariant Lie derivatives and FN bracket on Lie Algebroids]{Covariant Lie derivatives and Fr\"olicher-Nijenhuis bracket on Lie Algebroids}

\author[A. De Nicola]{Antonio De Nicola}
 \address{CMUC, Department of Mathematics, University of Coimbra, 3001-501 Coimbra, Portugal}
 \email{antondenicola@gmail.com}

\author[I. Yudin]{Ivan Yudin}
 \address{CMUC, Department of Mathematics, University of Coimbra, 3001-501 Coimbra, Portugal}
 \email{yudin@mat.uc.pt}

\thanks{Research partially supported by CMUC, funded by the European program
COMPETE/FEDER, by FCT (Portugal) grants PEst-C/MAT/UI0324/2011 (A.D.N. and I.Y.), by MICINN (Spain) grants
MTM2011-15725-E, MTM2012-34478 (A.D.N.), and exploratory research project in the frame of Programa Investigador FCT IF/00016/2013.}

\begin{document}
\maketitle

\begin{abstract}
We define covariant Lie derivatives acting on vector-valued forms on Lie algebroids and study their properties. This allows us to obtain  a concise formula for the Fr\"olicher-Nijenhuis bracket on Lie algebroids.
\end{abstract}

\section{Introduction}
The Fr\"olicher-Nijenhuis calculus was developed in the seminal
article~\cite{fn} and extended to Lie algebroids in~\cite{nijenhuis-alg}. It has proven to be an indispensable tool of Differential
Geometry. Indeed, different kinds of curvatures and  obstructions to integrability are computed by the Fr\"olicher-Nijenhuis bracket.
For example, if $J:TM\to TM$ is an almost-complex structure, then $J$ is complex
structure if and only if the Nijenhuis tensor $\mathcal{N}_J=\frac12 [J,J]_{FN}$ vanishes (this is the celebrated Newlander-Nirenberg theorem \cite{nn}). If $F:TM\to TM$ is a fibrewise diagonalizable endomorphism with real eigenvalues and of constant multiplicity, then  the eigenspaces of $F$ are integrable if and only if $[F,F]_{FN}=0$ (see \cite{haantjes}). Further, if $P:TE\to TE$ is a projection operator on the tangent spaces of a fibre bundle $E\to B$, then $[P,P]_{FN}$ is a version of the Riemann curvature (see \cite{kms}, page 78).
Finally, given a Lie algebroid $\alg$ and  $N\in \Gamma(\alg^*\otimes \alg)$
such that $[N,N]_{FN} =0$, one can construct a new (deformed) Lie algebroid $\alg_N$
(cf. \cite{grabowski97,YKSM}).
Moreover, Fr\"olicher-Nijenhuis calculus is useful in geometric mechanics where it allows to give an intrinsic formulation of Euler-Lagrange equations. In this field, Lie algebroids have also been shown to be a useful tool to deal with systems with some kinds of symmetries.

In \cite{michor-87}, P. Michor obtained a short expression for the Fr\"olicher-Nijenhuis bracket on manifolds in terms of the covariant Lie derivatives. A formula for the Fr\"olicher-Nijenhuis bracket on Lie algebroids in supergeometric language was obtained by P. Antunes in \cite{paulo2010}.
In this paper we define some operators relevant for Fr\"olicher-Nijenhuis
calculus in the setting of Lie algebroids, including the covariant Lie
derivative, and study their properties. In this way we are able to extend
Michor's formula for Fr\"olicher-Nijenhuis bracket to Lie algebroids.

\section{Covariant Lie derivative on Lie algebroids}
Let  $(\alg,[\;,\;],\anch)$ be a Lie algebroid over a manifold $M$, and $E$ a vector
bundle over~$M$. We write $\Omega^k(\alg,E)=\Gamma(\wedge^k \alg^*\otimes E)$ for the space of skew-symmetric
$E$-valued \mbox{$k$-forms} on $\alg$. If $E=M\times\R$ is the trivial line bundle over $M$, we denote $\Omega^k(\alg, E)$ by~$\Omega^k(\alg)$.

We write $\Sigma_m$ for the permutation group on $\left\{ 1,\dots, m \right\}$.
For $k$ and $s$ such that $k+s = m$, we denote by $\shuffle_{k,s}$ the subset of
$(k,s)$-shuffles in $\Sigma_m$.
Thus  $\sigma\in \shuffle_{k,s}$ if and only if
\begin{align*}
\sigma (1) < \sigma (2) <\dots <\sigma (k), && \sigma (k+1) <\dots<\sigma (k+s).
\end{align*}
Similarly, for a triple $(k,l,s)$, such that $k+l+s=m$, we denote by
$\shuffle_{k,l,s}$ the subset of $(k,l,s)$-shuffles in $\Sigma_m$, that is the set of
permutations $\sigma$, such that
\begin{align*}
	&\sigma(1) < \sigma(2) < \dots < \sigma(k), \  \quad
	\sigma(k+1) < \dots < \sigma(k+l),\\
	&\sigma(k+l+1) < \dots < \sigma(k+l+s).
\end{align*}

For a $k$-form $\omega\in \Omega^k (\alg)$ and $\phi\in \Omega^p(\alg, E)$ , we
define the form $\omega \wdd \phi\in \Omega^{k+p}(\alg, E)$ by
\begin{equation*}
	\left( \omega\wdd \phi \right)\left( Z_1,\dots, Z_{p+k} \right) =
	\sum_{\sigma\in\shuffle_{k,p}} (-1)^\sigma \omega\left( Z_{\sigma
	(1)},\dots,
	Z_{\sigma (k)} \right) \phi( Z_{\sigma (k+1)}, \dots, Z_{\sigma (k+p)}
).
\end{equation*}
Here and everywhere in this paper $Z_1,\dots, Z_{p+k}$  denote arbitrary sections of the Lie algebroid $\alg$.
If $E=M\times \R$ is the trivial line bundle over $M$, we denote $\wdd$ by
$\wedge$, and $\Omega^*(\alg)$ becomes a commutative graded algebra with the
multiplication given by $\wedge$. Further, note that $\Omega^*(\alg, E)$ is an
$\Omega^*(\alg)$-module with the action given by $\wdd$. For any
$\omega\in\Omega^k(\alg)$ we define the operator $\eps_\omega$ on
$\Omega^*(\alg, E)$ by
\begin{align*}
\eps_\omega : \Omega^*(\alg, E) &\to \Omega^{*+k}(\alg, E)\\
                        \phi&\mapsto \omega\wdd \phi
\end{align*}
Sometimes, given a operator $A$ we will use $\omega \wedge A$ as an alternative
notation for $\eps_\omega A$.

Let $\phi\in \Omega^p(\alg,\alg)$. For any vector bundle $E$ over $M$, we define the operator $i_\phi$ on
$\Omega^*(\alg,E)$ by
\begin{align}
	\label{iphi}
	\left( i_\phi\psi \right)\left( Z_1,\dots,Z_{p+k} \right) &=
	\!\!\!\!\!\sum_{\sigma\in \shuffle_{p,k}} \!\! \! (-1)^\sigma \psi\left( \phi(Z_{\sigma
	(1)},\dots, Z_{\sigma (p)}), Z_{\sigma (p+1)}, \dots, Z_{\sigma (p+k)} \right)
\end{align}
where $\psi\in \Omega^{k+1}(\alg,E)$.

We say that $\cov\colon \Gamma(\alg)\times \Gamma(E) \to\Gamma( E)$ is an
\emph{$\alg$-connection} on $E$ (see \cite{rui2002}) if
\begin{enumerate}[$1)$]
	\item $\cov_X$ is an $\R$-linear endomorphism of $\Gamma(E)$;
	\item $\cov s$ is a $\smooth(M)$-linear map from $\Gamma(\alg)$ to
		$\Gamma(E)$;
	\item $\cov_X( fs) = (\anch(X)f) s + f \cov_X s$ for any
$f\in \smooth(M)$, $X\in
		\Gamma(\alg)$, and $s\in \Gamma(E)$.
\end{enumerate}
The curvature of an $\alg$-connection $\cov$ is defined by
\begin{equation*}
	R(X,Y) s := \cov_X \cov_Y s - \cov_Y \cov_X s - \cov_{[X,Y]} s.
\end{equation*}
It is easy to check that $R$ is tensorial and skew-symmetric in the first two
arguments, thus we can consider $R$ as an element of $\Omega^2(\alg,
\End(E))$, where $\End(E)$ is the endomorphism bundle of $E$.

Given an $\alg$-connection on a vector bundle $E$, we define the covariant exterior
derivative on $\Omega^*(\alg, E)$ by
\begin{align*}
	(d^\cov \phi)\left( Z_1,\dots, Z_{p+1} \right) =&
	\sum_{\sigma \in \shuffle_{1,p}}(-1)^{\sigma}
	\cov^E_{Z_{\sigma(1)}}\left( \phi( Z_{\sigma(2)},\dots,
	Z_{\sigma(p+1)}
	)
	\right)\\& - \sum_{\sigma\in \shuffle_{2,p-1}}(-1)^{\sigma}\phi\left(
	\left[ Z_{\sigma(1)},Z_{\sigma(2)}
	\right],Z_{\sigma(3)},\dots,
	Z_{\sigma(p+1)} \right).
\end{align*}
Note that $d^\cov$ is related to the  curvature $R$ of $\cov^E$ by the formula
\begin{equation*}
	((d^\cov)^2 \phi)(Z_1, \dots, Z_{p+2}) =  \sum_{\sigma\in \shuffle_{2,p}}   (-1)^\sigma
	R(Z_{\sigma(1)},Z_{\sigma(2)}) \left(\phi (Z_{\sigma(3)},\dots, Z_{\sigma(p+2)})\right).
\end{equation*}
\begin{definition}
	A \emph{derivation} of degree $k$ on $\Omega^*(\alg, E)$ is a linear map
	$D\colon \Omega^*(\alg, E) \to \Omega^{*+k}(\alg,E)$
	such that
	\begin{equation*}
		D(\omega\wdd \phi) = \overline{ D }(\omega)\wdd\phi + (-1)^{kp} \omega\wdd
		D(\phi)
	\end{equation*}
	for all $\omega\in \Omega^p(\alg)$ and $\phi\in \Omega^*(\alg,E)$, where
	$\overline{ D }\colon \Omega^*(\alg) \to \Omega^*(\alg)$ is  some map.
\end{definition}
For any derivation $D$ on $\Omega^*(\alg, E)$ and $\alpha\in\Omega^*(\alg)$, we have
\begin{equation*}
[D, \eps_{\alpha}]=\eps_{\overline{ D } \alpha}.
\end{equation*}
In particular, the map $\overline{ D }$ is unique for a given derivation
$D$ on $\Omega^*(\alg,E)$. Let $\omega_1\in \Omega^{p_1}(\alg)$, $\omega_2\in
\Omega^{p_2} (\alg)$. From the following computation
\begin{align*}
	D( (\omega_1 &\wedge \omega_2 )\wdd \phi )  = \overline{ D
	}(\omega_1\wedge \omega_2) \wdd \phi + (-1)^{k(p_1 + p_2)} \omega_1
	\wedge \omega_2 \wdd D(\phi)\\
	D(\omega_1 &\wdd (\omega_2 \wdd \phi))  = \overline{ D }(\omega_1)\wedge
	\omega_2 \wdd \phi+
	(-1)^{kp_1} \omega_1 \wdd D(\omega_2 \wdd \phi)\\
	&= \overline{ D }(\omega_1)\wedge \omega_2 \wdd \phi + (-1)^{kp_1} \omega_1 \wedge \overline{
	D }(\omega_2) \wdd \phi + (-1)^{k(p_1+p_2)} \omega_1 \wedge \omega_2
	\wdd D(\phi)
\end{align*}
 one can see that $\overline{ D }$ is a derivation on $\Omega^*(\alg)$.

It is easy to check that for any given $\phi\in \Omega^k(\alg,\alg)$, $i_\phi$  is a
derivation of degree $k-1$, and $d^\cov$ is a derivation of degree $1$ on
$\Omega^*(\alg, E)$.
The \emph{covariant Lie derivative} $\lie^\cov_\phi$ is defined as the \emph{graded
commutator} $[i_\phi, d^\cov] = i_\phi d^\cov + (-1)^{k} d^\cov i_\phi$.
The graded commutator of two derivations of degree
$k$ and $l$ is a derivation of degree $k+l$. In particular, $\lie^\cov_\phi$ is
a derivation of degree $k$ for any $\phi\in \Omega^k(\alg,\alg)$.

Suppose
 we have an $\alg$-connection $\cov$ on $\alg$. We will say that $\cov$ is torsion-free if
$	\cov_X Y - \cov_Y X  = \left[ X,Y \right]$
for all $X$, $Y\in \Gamma(\alg)$.
On every algebroid $(\alg,[\;,\;],\anch)$, there exists a torsion-free
$\alg$-connection. Namely, one can take an arbitrary
bundle metric on $\alg$ and the associated Levi-Civita connection on $\alg$.
Given $\alg$-connections $\cov^\alg$ on $\alg$ and $\cov^E$ on $E$, we define
$\cov_X s\in \Omega^p(\alg,E)$ for every $s\in \Omega^p(\alg,E)$ by
\begin{equation*}
	(\cov_X s)(Z_1,\dots, Z_p):= \cov_X^E (s(Z_1,\dots,Z_p)) -
	\sum_{t=1}^p s(Z_1,\dots, \cov_X^\alg Z_t, \dots, Z_p).
\end{equation*}
It is easy to check that for any $s\in \Omega^k(\alg,E)$,
$X\in \Gamma(\alg)$, and a torsion-free $\alg$-connection on $\alg$,  we have
$	\lie_X^\cov s = \cov_X s + i_{\cov X }s $ and $\cov X = d^\cov
X$.
In other words
$
\cov_X = \lie_X^\cov - i_{d^\cov X}$.
Motivated by this relation, we define for $\phi\in \Omega^p(\alg,\alg)$ an operator  $\cov_\phi$ on
$\Omega^*(\alg, E)$ by
\begin{equation}
	\label{covphi}
	\cov_{\phi} := \lie^\cov_\phi -(-1)^p i_{d^\cov\phi}.
\end{equation}
Note that $\cov_\phi$ depends on two connections: an $\alg$-connection on
$E$ and a torsion-free $\alg$-connection on $\alg$.
Since $\cov_\phi$ is a linear combination of two derivations of degree
$p$, we see that $\cov_\phi$ is a derivation of degree $p$.
The following proposition shows that for $s\in \Omega^*(A,E)$
the map $\cov s\colon
\Omega^*(\alg,\alg) \to \Omega^*\left(\alg,E \right)$
is a homomorphism of $\Omega^*(\alg)$-modules.
\begin{proposition}
	\label{tensor}
For any $\omega\in \Omega^p(\alg)$, $\phi\in \Omega^k(\alg,\alg)$, and
$s\in \Omega^*(\alg,E)$, we have
\begin{equation*}
	\cov_{\omega\wdd\phi} s = (\omega \wedge \cov_\phi) s = \eps_\omega
	\cov_\phi s=  \omega \wdd
	(\cov_\phi s).
\end{equation*}
\end{proposition}
\begin{proof}
The equation
\begin{align*}
	\lie^\cov_{\omega\wdd \phi}& = \left[  i_{\omega\wdd \phi}, d^\cov \right] = \left[
	\omega\wedge i_\phi, d^\cov \right] = (-1)^{k+p}(d\omega)\wedge i_{\phi} + \omega\wedge
	\lie^\cov_\phi
\end{align*}
implies that
$\omega\wedge \lie^\cov_\phi =  \lie^\cov_{\omega\wdd \phi} -(-1)^{p+k} i_{(d\omega)\wdd
\phi}$.
	Now	we have
\begin{align*}
	\omega\wedge \cov_\phi &= \omega\wedge \lie^\cov_\phi -(-1)^p \omega \wedge
	i_{d^\cov \phi} = \lie^\cov_{\omega \wdd \phi} -(-1)^{p+k}i_{(d\omega)\wdd
	\phi} - (-1)^p  i_{\omega \wdd d^\cov \phi}
	\\&= \lie^\cov_{\omega\wdd \phi} -(-1)^{p+k} i_{d\omega \wdd \phi + (-1)^k
	\omega \wdd d^\cov \phi} = \cov_{\omega \wdd\phi}.
\end{align*}
\end{proof}
It was proven in \cite{nijenhuis-alg} that the commutator $[i_\phi,i_\psi]$ for $\phi\in
\Omega^k(\alg,\alg)$ and $\psi\in \Omega^l(\alg,\alg)$ is given by the formula
\begin{equation}
	\label{RN}
	[i_\phi,i_\psi] = i_{i_\phi \psi} - (-1)^{(k-1)(l-1)} i_{i_\psi \phi}.
\end{equation}
\begin{theorem}
	\label{icov-theorem}
Let $\cov$ be a torsion-free $\alg$-connection on $\alg$ and $\cov^E$ be an $\alg$-connection  on
a vector bundle $E$.
	For $\phi\in \Omega^k(\alg,\alg)$ and $\psi\in \Omega^l(\alg,\alg)$ we have
	on $\Omega^*(\alg,E)$
	\begin{equation}
		\label{icov-formula}
		[\cov_\phi, i_\psi] = i_{\cov_\phi \psi} - (-1)^{k(l-1)}
		\cov_{i_\psi \phi}.
	\end{equation}
\end{theorem}
\begin{proof}
	First we check the claim for $\phi =X \in \Gamma(\alg)$ and $\psi = Y
	\in \Gamma(\alg)$. Let $s\in \Omega^{p+1}(\alg, E)$.
	We get
	\begin{align*}
		(\cov_X i_Ys ) (Z_1,\dots, Z_p)& = \cov_X^E (s(Y,Z_1,\dots, Z_p)) -
		\sum_{t=1}^p s (Y, Z_1, \dots, \cov_X Z_t, \dots, Z_p)
		\\& = (\cov_X s)(Y,Z_1,\dots, Z_p) + s (\cov_X Y ,Z_1,\dots,
		Z_p) \\& = (i_Y \cov_X s)(Z_1,\dots, Z_p) + (i_{\cov_X Y}
		s)(Z_1,\dots, Z_p).
	\end{align*}
Thus
$	[\cov_X,i_Y] = i_{\cov_X Y}$.
Since \eqref{icov-formula} is additive in $\phi$ and $\psi$, it is enough to
prove it for $\phi = \alpha\wdd X$, $\psi  = \beta \wdd Y$,
where $\alpha\in \Omega^k(\alg)$, $\beta\in \Omega^l(\alg)$, and $X$, $Y\in
\Gamma(\alg)$. Repeatedly using Proposition~\ref{tensor} and $	[\cov_X,i_Y] = i_{\cov_X Y}$, we get
\begin{align*}
	\left[ \cov_{\alpha\wdd X}, i_{\beta \wdd Y} \right] & = \left[ \alpha
	\wedge \cov_X , \beta \wedge i_Y
	\right] = \left[ \eps_\alpha, \beta \wedge i_Y \right] \cov_X +
	\eps_\alpha \left[ \cov_X ,\beta \wedge i_Y \right] \\& =
	(-1)^{kl}\eps_\beta \left[ \eps_\alpha, i_Y \right] \cov_X + \eps_\alpha
	[\cov_X , \eps_\beta] i_Y + \eps_\alpha \eps_\beta [\cov_X, i_Y] \\ & =
	-(-1)^{kl-l} \eps_\beta \eps_{i_Y \alpha} \cov_X + \eps_\alpha \eps_{\cov_X
	\beta} i_Y + \eps_\alpha \eps_\beta i_{\cov_X Y}
	\\ & =
	i_{\alpha \wedge \cov_X \beta \wdd Y + \alpha \wedge \beta \wdd \cov_X Y } +
	(-1)^{(k-1)l} \cov_{\beta \wedge i_Y \alpha \wdd X}
	\\ & =
	i_{\alpha \wdd \cov_X (\beta \wdd Y)}
	+ (-1)^{(k-1)l} \cov_{\beta \wedge i_Y (\alpha \wdd X)}
	\\ & = i_{\cov_{\alpha \wdd X} ( \beta \wdd Y)} + (-1)^{(k-1)l} \cov_{
	i_{\beta \wdd Y} ( \alpha \wdd X)}.
\end{align*}
\end{proof}
To formulate the next result, we extend the definition of $R$ by defining for any $\phi\in \Omega^k(\alg,\alg)$ and $\psi\in \Omega^l(\alg,\alg)$ the form $R(\phi,\psi) \in \Omega^{k+l+1}(\alg,\alg)$
as follows
\begin{multline*}
	R(\phi,\psi)  (Y_1,\dots, Y_{k+l+1}) =\\=
	\sum_{\sigma\in \shuffle_{k,l,1}} R(\phi(Y_{\sigma(1)}, \dots,
	Y_{\sigma(p)}), \psi(Y_{\sigma(p+1)}, \dots, Y_{\sigma(p+q)}))
	Y_{\sigma(p+q+1)}.
\end{multline*}
\begin{theorem}
	\label{covcov-theorem}
	Let $\cov$ be a torsion-free $\alg$-connection on $\alg$ and $\cov^E$ a flat $\alg$-connection on a vector bundle $E$ over $M$ (i.e. $\cov^E$ is a representation of $\alg$).
	Then for any $\phi\in \Omega^k(\alg,\alg)$, $\psi\in \Omega^l(\alg,\alg)$, we have  the  following equality on $\Omega^*(\alg,E)$
\begin{equation}
	\label{covcov-formula}
	[\cov_\phi, \cov_\psi] = \cov_{\cov_\phi \psi} - (-1)^{kl}
	\cov_{\cov_\psi \phi} - i_{R(\phi, \psi)}.
\end{equation}
\end{theorem}
\begin{proof}
	First we prove \eqref{covcov-formula} for $\phi = X, \psi=Y \in \Gamma(\alg)$. For $s \in \Omega^p(\alg)$, we get
	\begin{align*}
(\cov_X \cov_Y& s) (Z_1,\dots, Z_p) \! =\! \cov_X^E ( \cov_Y^E s (Z_1,\dots, Z_p))
- \!\sum_{s=1}^p\! \cov_Y^E s (Z_1, \dots, \cov_X Z_s ,\dots, Z_p)
\\& = \cov_X^E\cov_Y^E (s(Z_1,\dots, Z_p)) - \sum_{s=1}^p  \cov_X^E(s(Z_1, \dots,
\cov_Y Z_s, \dots, Z_p)) \\& \phantom{=} - \sum_{s=1}^p\cov_Y^E (s(Z_1,\dots,
\cov_X Z_s, \dots, Z_p) + \sum_{s=1}^p s( Z_1, \dots, \cov_Y
\cov_X Z_s, \dots, Z_p) \\ &\phantom{=}
+ \sum_{s\not=t} s(Z_1, \dots, \cov_Y Z_t, \dots, \cov_X Z_s, \dots, Z_p).
	\end{align*}
	By anti-symmetrization of the above formula in $X$ and $Y$ and
	 using that $\cov^E$  is flat, we get
	\begin{align*}
		\left[ \cov_X,\!\! \cov_Y \right] s (Z_1,\dots, Z_p) &  =
		\!\!\cov_{[X,\!Y]}^E( s (Z_1,\dots,
Z_p))
-\!\! \sum_{s=1}^p s(Z_1, \dots, [\cov_X,\! \!\cov_Y] Z_s,
\dots, Z_p).
	\end{align*}
	Further
	\begin{align*}
		(\cov_{\cov_X Y} - \cov_{\cov_Y X}) s (Z_1,\dots, Z_p) & =
		\cov_{\cov_X Y - \cov_Y X}^E(s (Z_1,\dots, Z_p))
		\\& \phantom{=}
		- \sum_{s=1}^p s (Z_1, \dots, (\cov_{\cov_X Y} -
		\cov_{\cov_Y X})Z_s, \dots, Z_p).
	\end{align*}
	Taking the difference of  the last two formulas and using the definition of $R$ and
	that $\cov$  torsion-free, we have
	\begin{equation*}
		(([\cov_X,\!\!\cov_Y] - \cov_{\cov_X Y} + \cov_{\cov_Y X}) s)
		(Z_1,\dots, Z_p) = (-i_{R(X,Y)} s)(Z_1,\dots, Z_p).
	\end{equation*}
	Since \eqref{covcov-formula} is additive in $\phi$ and $\psi$, it is
	enough to prove it for $\phi = \alpha\wdd X$ and $\psi= \beta \wdd Y$,
	where $\alpha\in \Omega^k(\alg)$, $\beta \in \Omega^l(\alg)$, and
	$X$, $Y\in \Gamma(\alg)$. Using the already proved case and
	Proposition~\ref{tensor}, we get
	\begin{align*}
		[\cov_{\alpha\wdd X} ,\! \cov_{\beta \wdd Y}] & = [ \alpha
		\wedge \cov_X, \beta \wedge \cov_Y] = [\eps_\alpha, \beta \wedge
		\cov_Y] \cov_X + \eps_\alpha [\cov_X, \beta \wedge \cov_Y] \\& =
		(-1)^{kl}\eps_\beta [\eps_\alpha, \cov_Y] \cov_X + \eps_\alpha [\cov_X,
		\eps_\beta] \cov_Y + \eps_\alpha \eps_\beta [\cov_X, \cov_Y]
		 \\ & =
		 - (-1)^{kl}\eps_\beta \eps_{\cov_Y \alpha} \cov_X + \eps_\alpha
		\eps_{\cov_X \beta } \cov_Y + \eps_\alpha \eps_\beta
		(\cov_{\cov_X Y} - \cov_{\cov_Y X} - i_{R(X,Y)}).
	\end{align*}
	Repeatedly using Proposition~\ref{tensor}, we see that $[\cov_{\alpha\wdd X} ,\!
	\cov_{\beta \wdd Y}]$ can be written as  $ \cov_\theta + i_\tau $, where
	\begin{align*}
		\theta &= - (-1)^{kl}\beta \wedge \cov_Y\alpha \wdd X + \alpha \wedge \cov_X \beta \wdd Y +
\alpha \wedge \beta \wdd \cov_X Y - \alpha \wedge \beta \wdd \cov_Y X
\\& = \alpha \wdd\,\, \cov_X (\beta\, \wdd Y)   - (-1)^{kl} (\beta\, \wdd\,\, \cov_Y
(\alpha \wdd X))  = \cov_\phi \psi - (-1)^{kl} \cov_\psi \phi
	\end{align*}
	and
	\begin{equation*}
\tau = - \alpha \wedge \beta \wdd R(X,Y) = - R(\alpha \wdd X, \beta \wdd Y) = -
R(\phi, \psi).
	\end{equation*}
	This finishes the proof.
\end{proof}
Note that the connection $\cov_X^\rho f := \rho(X)f$ defined on the trivial line
bundle $M\times \R \to M$ is obviously flat. Thus \eqref{covcov-formula} holds
on $\Omega^*(\alg)$, if $\cov$ is defined via $\cov^\rho$ and any torsion-free
connection  on $\alg$.
\section{The Fr\"olicher-Nijenhuis bracket on Lie algebroids}
In \cite{nijenhuis-alg}, Nijenhuis defined the Fr\"olicher-Nijenhuis bracket on Lie algebroids of $\phi\in \Omega^k(\alg,\alg)$ and $\psi\in \Omega^l(\alg,\alg)$ by an equality of operators on $\Omega^*(\alg)$ equivalent to
\begin{equation}\label{def-fn}
	[\lie^\cov_{\phi},i_\psi] = i_{[\phi,\psi]_{FN}}- (-1)^{k(l-1)} \lie^\cov_{i_\psi \phi}.
\end{equation}
He also obtained a formula for computing $[\phi,\psi]_{FN}$.
In the next theorem we give an alternative formula using the covariant Lie derivatives, which extends the one obtained in \cite{michor-87} to the Lie algebroids setting.
\begin{theorem}
	\label{main}
Let $\phi\in \Omega^k(\alg,\alg)$ and $\psi\in \Omega^l(\alg,\alg)$. Suppose $\cov$ be a
torsion-free $\alg$-connection on $\alg$.  Then
\begin{equation*}
	[\phi,\psi]_{FN} = \lie^\cov_{\phi}\psi - (-1)^{kl} \lie^\cov_{\psi} \phi.
\end{equation*}
\end{theorem}
\begin{proof}
By \eqref{covphi} we have
\begin{align*}
	[\lie^\cov_{\phi},i_\psi] = [\cov_{\phi} + (-1)^k i_{d^\cov \phi},i_\psi] =[\cov_{\phi} , i_\psi] + (-1)^k [i_{d^\cov \phi},i_\psi].
\end{align*}
Hence, using \eqref{RN} and \eqref{icov-formula} we get
\begin{align*}
	[\lie^\cov_{\phi},i_\psi] = i_{\cov_\phi \psi} - (-1)^{k(l-1)}\nabla_{i_\psi \phi}+ (-1)^k i_{i_{d^\cov \phi} \psi} - (-1)^{kl}i_{i_\psi d^\cov \phi}.
\end{align*}
Next, using \eqref{covphi} in the second summand we have
\begin{align*}
	[\lie^\cov_{\phi},i_\psi] =& - (-1)^{k(l-1)}\left( \lie^\cov_{i_\psi \phi} -(-1)^{k+l-1} i_{d^\cov i_\psi \phi}  \right)\\
         &+  i_{\cov_\phi \psi} +(-1)^k i_{i_{d^\cov \phi} \psi} - (-1)^{kl}i_{i_\psi d^\cov \phi}.
\end{align*}
Notice that the subscripts of $\lie^\cov$ in \eqref{def-fn} and in the above formula are the same. Hence, due to the injectivity of $\phi\mapsto i_\phi$, we get by comparing the subscripts of $i$ that
\begin{align*}
         [\phi,\psi]_{FN}=&  (-1)^{k(l-1)}(-1)^{k+l-1} {d^\cov i_\psi \phi}+  {\cov_\phi \psi} +(-1)^k {i_{d^\cov \phi} \psi} - (-1)^{kl}{i_\psi d^\cov \phi}\\
         =&  {\cov_\phi \psi} +(-1)^k {i_{d^\cov \phi} \psi}
           - (-1)^{kl}({i_\psi d^\cov \phi} - (-1)^{l-1} {d^\cov i_\psi \phi} )\\
\end{align*}
Finally, using the definitions of $\cov_\phi$ and of $\lie^\cov_\psi$ we get the
claimed result.
\end{proof}
\bibliography{goldberg}
\bibliographystyle{plain}
\end{document}